\documentclass[12pt,oneside,english,shortlabels]{amsart}

\usepackage{babel}
\usepackage{amstext}
\usepackage{amssymb}

\usepackage{amsmath} 
\usepackage{latexsym}
\usepackage{graphicx} 

\usepackage{mathptmx}

\usepackage{float}

\RequirePackage{fix-cm}

\usepackage{xcolor}

\makeatletter

\usepackage{amsthm}
\usepackage{enumitem}

\usepackage{mathrsfs} 
\providecommand{\noopsort}[1]{}

\theoremstyle{definition} 

 \newtheorem{definition}{Definition}[section]
 \newtheorem{remark}[definition]{Remark}

\newtheorem*{notation}{Notations}

\theoremstyle{plain}      

 \newtheorem{proposition}[definition]{Proposition}
 \newtheorem{theorem}[definition]{Theorem}
 \newtheorem{corollary}[definition]{Corollary}

\makeatletter
\newcommand*{\house}[1]{
  \mathord{
    \mathpalette\@house{#1}
  }
}
\newcommand*{\@house}[2]{
  \dimen@=\fontdimen8 %
      \ifx#1\scriptscriptstyle\scriptscriptfont
      \else\ifx#1\scriptstyle\scriptfont
      \else\textfont\fi\fi
      3 %
  \sbox0{%
    $#1%
      \vrule width\dimen@\relax
      \overline{%
        \kern2\dimen@
        \begingroup % to keep changes of \dimen@ in #2 local
          #2%
        \endgroup
        \kern2\dimen@
      }
      \vrule width\dimen@\relax
      \mathsurround=1.5\dimen@ % outside margin
    $
  }
  \ht0=\dimexpr\ht0-\dimen@\relax
  \dp0=\dimexpr\dp0+2\dimen@\relax
  \vbox{
    \kern\dimen@ 
    \copy0 
  }
}

\newcommand{\rb}{\mathbb{R}}
\newcommand{\pb}{\mathbb{P}}

\newcommand{\cb}{\mathbb{C}}

\newcommand{\fb}{\mathbb{F}}

\newcommand{\zb}{\mathbb{Z}}

\newcommand{\qb}{\mathbb{Q}}
\newcommand{\nb}{\mathbb{N}}

\def\11{{\mathbf 1}}

\theoremstyle{remark}
\newtheorem{exampl}[subsubsection]{Example}

\def\bee{\begin{exampl}}
\def\eee{\end{exampl}}
\def\bn{\begin{notation}}
\def\en{\end{notation}}
\def\br{\begin{remark}}
\def\er{\end{remark}}
\def\bp{\begin{prop}}
\def\ep{\end{prop}}
\def\bpr{\begin{proof}}
\def\epr{\end{proof}}
\def\bt{\begin{thm}}
\def\et{\end{thm}}
\def\be{\begin{equation}}
\def\ee{\end{equation}}
\def\bl{\begin{lem}}
\def\el{\end{lem}}
\def\bc{\begin{cor}}
\def\ec{\end{cor}}
\def\bd{\begin{defn}}
\def\ed{\end{defn}}

\author{Denys Dutykh$^{\ddag}$}
\thanks{}
\address{$^{\ddag}$
Univ. Grenoble Alpes, Univ. Savoie Mont~Blanc,
LAMA, CNRS UMR 5127,
\newline
F - \!73000 Chamb\'ery, France}
\email{denys.dutykh@univ-smb.fr}
\urladdr{}

\author{Jean-Louis Verger-Gaugry$\mbox{}^{\Diamond}$ }
\thanks{}
\address{$\mbox{}^{\Diamond}$ 
Univ. Grenoble Alpes, Univ. Savoie Mont~Blanc,
LAMA, CNRS UMR 5127,
\newline
F - \!73000 Chamb\'ery, France}
\email{Jean-Louis.Verger-Gaugry@univ-smb.fr}
\urladdr{}

\title[On a class of lacunary
almost Newman polynomials modulo $p$ and
density
theorems]
{On a class of lacunary
almost Newman polynomials modulo $p$ and
density
theorems
}

\begin{document}

\title{On a class of lacunary
almost Newman polynomials modulo $p$ and
density
theorems}

\begin{abstract}
The reduction modulo $p$ of
a family of lacunary integer polynomials,
associated with the dynamical zeta function
$\zeta_{\beta}(z)$ of the 
$\beta$-shift, for
$\beta > 1$ close to one, 
is investigated.
We briefly recall how this family
\textcolor{black}{is} correlated to the problem of Lehmer.
A variety of questions is raised about
\textcolor{black}{their numbers of zeroes in $\fb_p$
and}
their
\textcolor{black}{
factorizations,} 
via
Kronecker's Average Value Theorem
(viewed as an analog of
classical Theorems of Uniform Distribution Theory).
\textcolor{black}{These questions are} partially answered
using results of Schinzel, revisited by
Sawin, Shusterman and Stoll, and 
density theorems (Frobenius, 
Chebotarev, Serre, Rosen). 
These questions arise from the 
\textcolor{black}{
search for}
the existence of integer polynomials of Mahler measure
$> 1$ less than the smallest Salem number $1.176280$.   
Explicit connection with modular forms 
(or modular representations) of the numbers of zeroes
\textcolor{black}{of these polynomials}
in $\fb_p$ is \textcolor{black}{obtained in a few cases.
In general it is expected
since it must exist 
according to} the Langlands program. 
\end{abstract}

\maketitle

\vspace{0.7cm}


\vspace{0.5cm}


\tableofcontents

\numberwithin{equation}{section} \numberwithin{figure}{section}
\newpage
\section{Introduction}
\label{S1}

In this paper we will address a
number of arithmetic questions
on a class of lacunary integer polynomials
which have all their coefficients in 
$\{0, 1\}$ except the constant term
equal to $-1$. We have called
such polynomials 
{\it almost Newman polynomials}
in \cite{dutykhvergergaugry18}, by comparison with
Newman polynomials which all
have their coefficients
in $\{0, 1\}$. 
\textcolor{black}{Recall that, with probability one,
Newman polynomials are irreducible 
\cite{breuillardvarju}}.

This class is the following.
For $n \geq 2$,
we denote by $\mathcal{B}$ the class of 
lacunary polynomials
$$f(x) := -1 + x + x^n +
x^{m_1} + x^{m_2} + \ldots + x^{m_s}$$
where $s \geq 0$, $m_1 - n \geq n-1$, 
$m_{q+1}-m_q \geq n-1$ for $1 \leq q < s$,
and by
$\mathcal{B}_n$ those whose third monomial is exactly
$x^n$, so that
$$\mathcal{B} =
\cup_{n \geq 2} \mathcal{B}_n .$$
The case ``$s=0$" corresponds to 
the trinomials $G_{n}(z) :=
-1+z+z^n$.
This class
admits a very special type of lacunarity.

This class
appears naturally in the study of 
the existence of reciprocal \textcolor{black}{integer}
 polynomials
having small Mahler measure
\cite{smyth}: 
\textcolor{black}{in the Appendix we recall how
this class of polynomials is arising from
the dynamical zeta function $\zeta_{\beta}(z)$
of the $\beta$-shift (equivalently from the Parry Upper function $f_{\beta}(z)$), 
when $\beta > 1$ is close to one.}
\textcolor{black}
{Recall that the Mahler measure
of a nonzero algebraic number $\beta$,
of minimal polynomial
$P_{\beta}(X)=a_0 X^m + a_1 X^{m-1}
+\ldots + a_m = a_0 \prod_i (X-\alpha^{(i)})
\in \zb[X]$, is
${\rm M}(\beta) = |a_0|\prod_i \max\{1, |\alpha^{(i)}|\} 
=:
{\rm M}(P_{\beta})$.}
\textcolor{black}{
The search for non-trivial} Mahler measures of 
reciprocal integer polynomials smaller
than the smallest Salem number known $1.176280\ldots$
\textcolor{black}{(Lehmer's number)}
is of particular interest in the Problem of Lehmer
\cite{smyth}
\cite{vergergaugryP}.
\textcolor{black}{
When $\beta > 1$ is close to one,
it is shown
in \cite{dutykhvergergaugry21} 
\cite{vergergaugryv8}
that a subcollection of zeroes
(the lenticular zeroes, cf Section \ref{S2}) 
of $f_{\beta}(z)$ is at the origin
of a nontrivial (universal) minorant of
${\rm M}(\beta)$. 
By Hurwitz Theorem,
the zeroes of $f_{\beta}(z)$ 
in the open unit disk of $\cb$ are limit points
of zeroes of polynomials of the class 
$\mathcal{B}$, which are its polynomial sections.
It is the reason why this class of lacunary
almost Newman polynomials is important
for the Problem of Lehmer.
}

\

\textcolor{black}{
By analogy, there is interest in studying the 
zeroes of $f(z)$
modulo $p$, for any $f \in \mathcal{B}$
and any prime number $p$.}
In the present note, we start the investigation of
(i) 
the number of zeroes $N_{p}(f)$
of the polynomials
$f \in \mathcal{B}$ 
in $\fb_p$, i.e.
modulo a
prime number $p$, as a function of $p$, $n$,
and the type of lacunarity of $f$ characterized
by the sequence $(m_j)_{j=1,\ldots,s}$,
(ii) 
the subcollection of prime numbers $p$
for which $N_{p}(f)$ is equal to zero or is maximal,
(iii) 
the asymptotics 
\textcolor{black}{of the averages of} $N_{p}(f)$ when 
$p$ tends to infinity.

\

\textcolor{black}{
For a nonzero integer polynomial
$f(X) \in \zb[X]$ we denote by
$c(f)$ the greatest common divisor of its
coefficients. 
Let $\zb[X]^c : = \{f \in \zb[X] \setminus \{0\}:
c(f) =1\}$.
Let $\mathcal{N}$ be the map
$\zb[X]^c \to (\nb)^{\pb},
f \to (N_{p}(f))_{p \in \pb}$.
Obviously, for
$f_{1}, f_2 \in \zb[X]^c$,
 $\mathcal{N}(f_1 f_2)=
\mathcal{N}(f_1) + \mathcal{N}(f_2)$
since, for any $p \in \pb$,
$N_{p}(f_1 f_2) = N_{p}(f_1) +
N_{p}(f_2)$. 
Now, integrating 
$\mathcal{N}$ over the set of prime numbers
$\pb$ and taking the limit average with respect to
$\pi(x)$, which is as usual 
the number of primes  $p \leq x$, gives the following result (proof in Section \ref{S3}).}

\newpage
\textcolor{black}{
\begin{theorem}
\label{mainN}
Let $f \in \zb[X]^c$. 
If $f = \prod _{i \in J} f_{i}^{\nu_i}$
is the decomposition of $f$ into irreducible factors,
with all $\nu_i \geq 1$, then
\begin{equation}
\label{kroneckeraverage}
\lim_{\textcolor{black}{x} \to \infty}
\frac{1}{\pi(x)}\sum_{p \leq x}
N_{p}(f)
=
\sum_{i \in J} \nu_i
\end{equation}
which is the number of irreducible factors
in the decomposition of $f$. 
\end{theorem}
}
\textcolor{black}{
By the Langlands program closed formulas are expected
between the values $N_{p}(f)$
and the coefficients $a(n)$ of Newforms
$\sum a(n) q^n$. 
Key properties of Newforms, geometrical objects
attached to the coefficients,
can be found e.g. in
\cite{cohenstromberg}
\cite{ono}.
}

\textcolor{black}{
Let us observe that \eqref{kroneckeraverage}
is an analog of classical Theorems of Uniform Distribution \cite{kuipersniederreiter}
\cite{strauch} (cf Section \ref{S3}).
}

\textcolor{black}{
Therefore the present study on the class
$\mathcal{B}$ starts a concrete investigation
not only of the quantities 
$N_{p}(f), f \in \mathcal{B}$, and
the associated Newforms,
but also of their
limit averages by \eqref{kroneckeraverage},
which corresponds to the factorization of $f$.
}

\textcolor{black}{
\begin{remark}
In the 
literature, 
following Lehmer's work,
there are two ``Conjectures of Lehmer", 
which a priori are completely independent. 
The first one
is evoked in 
\cite{smyth} \cite{vergergaugryP}, and is concerned
by a universal minorant of the Mahler measure
of reciprocal nonzero algebraic 
integers which are not roots of unity. 
It is of concern in this study.
The second one is related to
the $\tau$-function of Ramanujan.
The present work, though trying to correlate
the Problem of Lehmer modulo $p$ to modular forms
(Newforms) has no objectives
to try to link the two Conjectures. The authors 
have no idea if the two Conjectures are linked. 
\end{remark}} 

The factorization of $f \in \mathcal{B}$
as a function of its reductions modulo $p$
is a deep question
\cite{lenstrastevenhagen}
\cite{stevenhagenlenstra}. 
It is well-known that a monic integer polynomial is irreducible over the rationals if 
it is irreducible modulo some prime.
The converse is not true in general
\cite{brandl}.
\textcolor{black}{Even if} 
$f \in \mathcal{B}$ is irreducible a first question is about the density of 
the set of prime numbers
$p$ such that $N_{p}(f) = 0$.
Indeed Guralnick, Schacher and Sonn
\cite{guralnickschachersonn}, then
Gupta \cite{gupta}, have shown the existence
of irreducible integer polynomials which
are reducible modulo all primes.
Whether the factorization of $f$ modulo $p$
contains linear factors is a basic question.
If it is so, it is natural to say that
the number 
$N_{p}(f)$ can take a priori
any value between 1 and
the maximal value $\deg (f)$.

By Conjecture ARC 
(Conjecture \eqref{ARCconjecture} below) 
only 75\% of
the polynomials $f$ in $\mathcal{B}$ are
irreducible. 
General irreducibility criteria for
the polynomials
$f$ in $\mathcal{B}$ are missing.
Theorem \ref{mainN} (ii)
(Kronecker's Average Value Theorem), 
applied to each irreducible $f \in \mathcal{B}$,
says
$$
\lim_{p \to \infty}
\frac{1}{\pi(x)}\sum_{p \leq x}
N_{p}(f)
=
1,$$
preventing the values $N_{p}(f) = \deg(f)$
to occur often, since
the sum $\sum_{p \leq x}
N_{p}(f)$ behaves like $x/\log(x)$ at infinity.
If $f \in \mathcal{B}$ is irreducible and
$\deg(f)$ is a prime number, then $\mathcal{N}(f)
= (N_{p}(f))_{p \in \pb}$ is such that infinitely many
$N_{p}(f)$ are equal to 0 \cite{stevenhagenlenstra}.
Examples of the distribution of
values $N_{p}(S_j)$ 
of the (non-reciprocal parts of the) polynomial sections
$S_{j}(x)$ of the Parry Upper function
$f_{\tau}(x) = -1/\zeta_{\tau}(x)$,
where $\tau$ is Lehmer's number
(cf Section \ref{S2} for its definition),
are given in Section \ref{S5}.

\

\textcolor{black}{
The density Theorems of Frobenius and Chebotarev
\cite{chebotarev}
play a role in the frequencies of values
taken by $N_{p}(f)$ as $p$ varies in general.
For the class $\mathcal{B}$}
we summarize this in Theorem \ref{densitythm}
below. 
The question
of the non-existence of factors of 
degree 1 (the ``$N_{p}(f) =0$" case)
in the factorization of $f$
modulo $p$, 
is partially answered by a general theorem of Serre
\textcolor{black}{(Serre's Density Theorem 
\cite{serre})},
which gives a positive 
lower bound on the density of such primes
(reported 
in Theorem \ref{densitythm} (i)).

The density is taken in the following sense:
a subset $S$ of the set of primes 
$\mathbb{P}$ has density 
$c$ if
$$\lim_{x \to \infty}
\frac{\mbox{number of}~ p \in S ~\mbox{with}~ p \leq x}
{\pi(x)}= c.
$$ 
The limit need not exist.
When it exists the natural density of 
$S$ is said to be defined; then it is
denoted by $\delta(S)$.

\begin{theorem}
\label{densitythm}
Let $n \geq 3$. Let
$$f(x) := -1 + x + x^n +
x^{m_1} + x^{m_2} + \ldots + x^{m_s} \quad
\in \mathcal{B}$$
where $s \geq 0$, $m_1 - n \geq n-1$, 
$m_{q+1}-m_q \geq n-1$ for $1 \leq q < s$.
The polynomial $f$ is assumed irreducible.
Denote by G the Galois group of $f(x)$
and $g := \#G$.
Then 
\begin{enumerate}
\item[(i)]
the set 
$\mathcal{P}_0 :=\{p \in \mathbb{P}
\mid N_{p}(f) = 0\}$
is infinite, has a density and its density
satisfies
$$\delta(\mathcal{P}_0) \geq \frac{1}{m_s}, $$
with strict
inequality if $m_s$ is not a power of a prime,
\item[(ii)]
the set 
$\mathcal{P}_{max} :=\{p \in \mathbb{P}
\mid N_{p}(f) = \deg(f)\}$
is infinite and has density
$$\delta(\mathcal{P}_{max}) = \frac{1}{g}. $$
\end{enumerate}
\end{theorem}

Theorem \ref{densitythm} (i) is due to Serre; 
the proof of (i) is given  in \cite{serre}. 
The statement of (ii) is Corollary 2 in
Rosen \cite{rosen}. 
This Corollary 2 is a consequence
of a Theorem of Frobenius 
(Theorem 2 in \cite{rosen}).

\textcolor{black}{In Section \ref{S4}} we  
continue the direct study
of the factorization of the polynomials
$f$ in the class $\mathcal{B}$, initiated
in \cite{dutykhvergergaugry18}.
Indeed this study has left open
the problem of the existence of 
reciprocal non-cyclotomic components in the
factoring of any such $f$.
The main theorem
on the factorization of all
$f \in \mathcal{B}$
is Theorem
\ref{thm1factorization}
\cite{dutykhvergergaugry18}.
We recall in Section \ref{S2} 
the Asymptotic Reducibility Conjecture 
(``ARC") which states that
the probability of
finding a polynomial
$f$ in $\mathcal{B}$ which is irreducible
is $3/4$, and
 the Conjecture ``B", which states
the non-existence
of a reciprocal non-cyclotomic
component in the factorization of
a polynomial $f$ in $\mathcal{B}$.
We revisit Conjecture B 
using a Theorem of Schinzel and
a recent new lower bound by
Sawin, Shusterman and Stoll
\cite{sawinshustermanstoll} for large gaps.
We prove that Conjecture B 
is valid on some infinite
subclasses of $\mathcal{B}$.
For doing this,
we consider the
new bound for large gaps given in 
\cite{sawinshustermanstoll} as 
a new critical value above which  
Conjecture B is always true.
Then, at intermediate lacunarity,
for moderate gaps below this critical value,
we show numerically that Conjecture B is 
also true.
The subclasses considered are
families of
pentanomials
in $\mathcal{B}$, chosen from
the quadrinomials studied by
Finch and Jones \cite{finchjones}.

\

The 
Problem of Lehmer
modulo $p$, adressed to the class $\mathcal{B}$,
with its asymptotics when 
$p$ tends to infinity, also
calls for 
understanding the interplay between
the asymptotics of $N_{p}(f)$
and the peculiar lacunarity of $f$, for any
$f \in \mathcal{B}$.
The importance of the problem of the factorization of lacunary polynomials was outlined
in a series of papers by
Schinzel \cite{schinzel69}
\cite{schinzel76}
\cite{schinzel78}
\cite{schinzel83}
\cite{schinzel00}.
The link between lacunarity, say the 
geometry of the gappiness, 
and irreducibility
of any
$f \in \mathcal{B}$, 
%once $f$ is assumed irreducible, 
is emphasized
in Corollary 1.4 in
\cite{sawinshustermanstoll}, as follows.

\begin{theorem}\cite{sawinshustermanstoll}
\label{thmprogressions}
Let $f \in \mathcal{B}$,
and write it as a
polynomial $f(x) =
g_{N}(x) =
d(x) + x^N c(x^{-1})$ as
in \eqref{gNdefinition} 
\textcolor{black}{and}
\eqref{fgN}.
Under the assumptions of Schinzel's Theorem 
\ref{schinzelthm1_2},
the set  of 
$N > \deg c + \deg d$ such that
$f$ is irreducible is the complement of the union
of a finite set with a finite union of
arithmetic progressions.
\end{theorem}

\textcolor{black}{In Section \ref{S6}} the quantities
$N_{p}(f)$, $p$ tending to infinity,
are studied for trinomials, 
following \cite{serre},
as functions of
the coefficients
of $q$-expansions and correlated to Newforms
and modular forms. 
Some basic questions
about the densities 
of $p$s such that $N_{p}(f)$ is congruent
to a fixed integer modulo some integers
are asked in Serre's general context
\cite{serreVID1} \cite{serreVID2}.

\textcolor{black}{
To outline the novel strategy of the paper,
this note initiates the study of
the quantities
$N_{p}(f)$, $p$ any prime number,
and also $p$ tending to infinity,
for any $f$ in the class $\mathcal{B}$,
and concomittantly the factorization of such
$f$s. This difficult problem is
tackled by the simplest cases of $f$s in
$\mathcal{B}$, which are trinomials,
as in Section \ref{S6}; and by
showing on some subfamilies of $\mathcal{B}$, 
as in Section \ref{S4}, that
factorization occurs with
the absence of reciprocal non-cyclotomic 
components, which is 
conjectured to be the general rule.}

\section{Dynamical zeta function of the $\beta$-shift and Lehmer's problem}
\label{S2}

Let us recall standard definitions.
A complex number $\alpha$ is an algebraic
integer if there exists a monic polynomial
$R(X) \in \zb[X]$ such that
$R(\alpha)=0$. If $R$ is the minimal
polynomial of $\alpha$ and is reciprocal, i.e.
satisfies 
$X^{\deg R} R(1/X)= R(X)$,
then $\alpha$ is \textcolor{black}{called} reciprocal.
If $\alpha$ is reciprocal, $\alpha$ and
$1/\alpha$ are conjugated.
If the minimal polynomial of $\alpha$
is not reciprocal, $\alpha$ is 
called non-reciprocal.
If $\alpha = 1$, or if
$\alpha > 1$
and the conjugates $\alpha^{(i)} \neq \alpha$ of
$\alpha$ satisfy:
$|\alpha^{(i)}| < \alpha$, then
$\alpha$ is said to be a Perron number.

Let $n \geq 3$ be a fixed integer.
Selmer \cite{selmer} has shown that 
the trinomials 
$-1+x+x^n \in \zb[x]$ are irreducible
if $n \not\equiv 5 ~({\rm mod}~ 6)$, and, for
$n \equiv 5 ~({\rm mod}~ 6)$, are reducible as 
product of two irreducible factors whose
one is the cyclotomic factor $x^2 -x +1$,
the other factor 
$(-1 + x + x^n)/(x^2 - x +1)$
being nonreciprocal of degree
$n-2$.
We denote by $\theta_n$ the unique zero
in $(0,1)$ of the trinomial
$-1+x+x^n$. 
The inverses $\theta_{n}^{-1} > 1$ 
are non-reciprocal algebraic integers
which are Perron numbers,
and constitute a decreasing sequence
$(\theta_{n}^{-1})_{n \geq 3}$ 
tending to $1^+$. 
The smallest Mahler measure known
is Lehmer's number $\tau$, 
root $> 1$ of 
$x^{10}+x^{9}-x^7-x^6-x^5-x^4-x^3 +x+1$,
such that:
$\theta_{12}^{-1} < \tau = 1.176280\ldots
< \theta_{11}^{-1}$.  
Here $n$ is equal to $12$.
The 
\textcolor{black}
{search for} reciprocal algebraic integers 
$\beta$s,
of Mahler measure
M$(\beta) \leq 1.176280$,
in the intervals
$(\theta_{n}^{-1} , 
\theta_{n-1}^{-1})$, $n \geq 13$,
having a minimal polynomial for which there is
no $\zb$-minimal integer polynomial  
$\widehat{P_{\beta}}(X)$
such that
\begin{equation}
\label{requalityPP}
P_{\beta}(X) = \widehat{P_{\beta}}(X^r)
\end{equation}
for some integer $r \geq 2$ 
is of importance
in the problem of the non-trivial 
minoration of the Mahler measure 
%and therefore of the height of rational points
%\cite{laurent} 
(Section 2 in \cite{vergergaugryP}).
The existence of very small Mahler measures
is still a mystery.

Let us assume the 
existence of such a reciprocal  
algebraic integer
$\beta > 1$. 
It is canonically associated with, and
characterized by,
two analytic functions:
\begin{enumerate}
\item[(i)]
its minimal polynomial function, say
$z \to P_{\beta}(z)$, which is monic and
reciprocal; 
denote $d := \deg P_{\beta}$,
$H:=$ the (na\"ive) height of $P_{\beta}$,
\item[(ii)] the Parry Upper function 
$f_{\beta}(x)$ at
$\beta^{-1}$, which is the
generalized Fredholm
determinant of the $\beta$-transformation $T_{\beta}$
(Section 3 in \cite{vergergaugryv8}).
It
is a (infinite) power series
with coefficients in the alphabet $\{0,1\}$
except the constant term equal to $-1$,
with distanciation between the exponents of the monomials:
$$f_{\beta}(x) := -1 + x + x^n +
x^{m_1} + x^{m_2} + \ldots + x^{m_s} +
\ldots$$
where $m_1 - n \geq n-1$, 
$m_{q+1}-m_q \geq n-1$ for $q \geq 1$.
$\beta^{-1}$ is the unique 
zero of $f_{\beta}(x)$
in the unit interval $(0,1)$.
The analytic function
$f_{\beta}(z)$ is related to the dynamical
zeta function $\zeta_{\beta}(z)$ of the
$\beta$-shift 
(Section 3 in \cite{vergergaugryv8};
\cite{flattolagariaspoonen}) by:
$f_{\beta}(z) = -1/\zeta_{\beta}(z)$.
Since $\beta$ is 
reciprocal, with the two real roots
$\beta$ and $1/\beta$, the series
$f_{\beta}(x)$ is never a polynomial,
by Descartes's rule on sign changes on the 
coefficient vector.
The algebraic integer $\beta$ is 
associated with the infinite
sequence of exponents $(m_j)$.  
\end{enumerate}

Let us observe that 
all the polynomial sections 
$$S_{s}(x)
:=-1 + x + x^n +
x^{m_1} + x^{m_2} + \ldots + x^{m_s}, \qquad
s \geq 1,$$
of $f_{\beta}(x)$
are polynomials of the class
$\mathcal{B}_n$.

The principal motivation
to study the class $\mathcal{B}$ for itself in the present note
comes
from the peculiar form of these polynomial sections.

The polynomials $f$ of the class
$\mathcal{B}$ are
often
irreducible by the following conjecture, 
formulated
in \cite{dutykhvergergaugry18}.
 
\vspace{0.3cm}

\noindent
{\it {\bf Asymptotic Reducibility Conjecture (ARC)}}
~
{\it 
Let $n \geq 2$ and
$N \geq n$. Let
$\mathcal{B}_{n}^{(N)}$ denote
the set of the polynomials
$f \in \mathcal{B}_{n}$ 
such that $\deg(f) \leq N$.
Let
$\mathcal{B}^{(N)} 
:= \bigcup_{2 \leq n \leq N} 
\mathcal{B}_{n}^{(N)}$.
The proportion of
polynomials in
$\mathcal{B} =
\cup_{N \geq 2} \mathcal{B}^{(N)}$
which are irreducible is given by the limit, assumed to exist,
\begin{equation}
\label{ARCconjecture}
\lim_{N \to \infty} 
\frac{\#\{ f \in \mathcal{B}^{(N)} \mid f~ \mbox{irreducible} \}}
{\#\{ f \in \mathcal{B}^{(N)}\}}
\qquad \mbox{and its value is expected to be}
~~\frac{3}{4}.
\end{equation}
}

Let us recall the generic factorization of 
the polynomials
$f \in \mathcal{B}$ 
\cite{dutykhvergergaugry18} which 
generalizes that of the
trinomials $-1 +x+x^n$ by Selmer \cite{selmer}.

\begin{theorem}[Dutykh Verger-Gaugry]
\label{thm1factorization}
For any $f \in \mathcal{B}_n$, $n \geq 3$,
denote by
$$f(x) = A(x) B(x) C(x) =
-1 + x + x^n +
x^{m_1} + x^{m_2} + \ldots + x^{m_s},$$
where $s \geq 1$, $m_1 - n \geq n-1$, 
$m_{j+1}-m_j \geq n-1$ for $1 \leq j < s$,
the factorization of  $f$ where $A$ is 
the cyclotomic part, 
$B$ the reciprocal noncyclotomic part,
$C$ the nonreciprocal part.
Then (i)
the  nonreciprocal part $C$ is
nontrivial, irreducible, 
and never vanishes on the unit circle,
(ii) if $\beta > 1$
denotes the real algebraic integer
uniquely determined 
by the sequence
$(n, m_1 , m_2 , \ldots , m_s)$ such that
$1/\beta$ is the unique
real root of $f$ in
$(\theta_{n-1} , \theta_{n})$,
the nonreciprocal polynomial
$-C^*(X)$ of
$C(X)$
is the minimal polynomial 
of $\beta$, and $\beta$ is a nonreciprocal
algebraic integer.
\end{theorem}
From numerous experiments on the class
$\mathcal{B}$ by Monte-Carlo
in \cite{dutykhvergergaugry18}, 
the components ``B" were never observed.
The following conjecture is reasonable 
to formulate. It will be partially proved
in Section \ref{S4}.
\vspace{0.2cm}

\noindent
{\it {\bf Conjecture B}}~
{\it The reciprocal non-cyclotomic part B 
of any $f \in \mathcal{B}_n$,
$n \geq 3$, is always trivial.}
\vspace{0.2cm}

\noindent

Let us mention now the link
with Lehmer's problem. Lehmer's number
$1.176280$ is the smallest Mahler 
measure known. By a theorem of Smyth 
\cite{smyth}
the Mahler measure of a nonzero 
algebraic integer
which is not reciprocal, not a root of unity,
is $\geq 1.3247...$, dominant root
of $X^3 -X-1$ and smallest Pisot number.
Then the Mahler measures of algebraic integers
which are in the range $(1, 1.3247)$
arise from
reciprocal algebraic integers which are not roots of unity.
This is 
\textcolor{black}{the main reason} 
why Conjecture B is important to investigate, about the possible
existence
of reciprocal parts in the factors
of the polynomials of the class $\mathcal{B}$.

\textcolor{black}{Our attention is focused}
on the 
\textcolor{black}
{search for} 
hypothetical
reciprocal algebraic integers $\beta > 1$
for which M$(\beta) \in (1, 1.176280)$
and \eqref{requalityPP} is satisfied, 
\textcolor{black}{that is when $n$ is large in 
Conjecture B.  }
Using intermediate alphabets and 
periodic representations of 
$\qb(\beta)$
in the algebraic basis $\beta$, 
it was shown in
\cite{dutykhvergergaugry21} that
the relation between 
$f_{\beta}$ and
$P_{\beta}$ is a relation of identification
on the subcollection of lenticular
zeroes of $f_{\beta}$. 
\textcolor{black}{
The definition
of a lenticular zero is
given in \cite{dutykhvergergaugry18},
where many examples are proposed.}
\vspace{1cm}

\textcolor{black}{
The lenticular zeroes of
$f_{\beta}$ are peculiar zeroes, off the unit circle.
Let us briefly recall what is a lenticular zero
of $f_{\beta}$. Many examples of
lenticular zeroes are given in
\cite{dutykhvergergaugry18}. 
The following theorem is Theorem 4 in
 \cite{dutykhvergergaugry18}. 
}
\textcolor{black}{
\begin{theorem}
\label{thm2lenticuli}
Assume $n \geq 260$.
There exist
two positive constants $c_n$ and $c_{A,n}$ , 
$c_{A,n} < c_n$,
such that the
roots of any
$f \in \mathcal{B}_n$,
$$f(x)=
-1 + x + x^n +
x^{m_1} + x^{m_2} + \ldots + x^{m_s},$$
where $s \geq 1$, $m_1 - n \geq n-1$, 
$m_{j+1}-m_j \geq n-1$ for $1 \leq j < s$,
lying in
$- \pi/18 < \arg z < + \pi/18$
either belong to
$$\{z \in \cb : ||z|-1| < \frac{c_{A,n}}{n}\},\quad
\mbox{or to}\quad
\{z \in \cb : ||z|-1| \geq \frac{c_n}{n}\}.$$
\end{theorem}
}
\textcolor{black}{
The {\em lenticulus of zeroes $\omega$ 
of $f$} is then defined
as
$$\mathcal{L}_{\beta} :=
\{\omega \in \cb : f(\omega) =0, |\omega| < 1,
-\frac{\pi}{18} < \arg \omega < +\frac{\pi}{18},~
||\omega|-1| \geq \frac{c_n}{n}\}$$
where $1/\beta \in \mathcal{L}_{\beta}$ 
is the positive real
zero of $f$. If a zero of 
$f$ belongs to $\mathcal{L}_{\beta}$ 
we say that it is a 
{\em lenticular zero of $f$}.
}

\vspace{1cm}

More precisely
if $\Omega$ is a lenticular zero, then
\textcolor{black}{$$f_{\beta}(\Omega) = 0
\Longrightarrow
P_{\beta}(\Omega)=0.$$
}

\textcolor{black}{
\section{An analog of Uniform Distribution Theorems - Proof of Theorem \ref{mainN}}
\label{S3}
}
\textcolor{black}{
Recall \cite{kuipersniederreiter}
\cite{strauch} that we say a sequence 
$(x_n)_{n\geq 1} \subseteq [0,1)^s$ is 
\it uniformly distributed  on $[0,1)^s$ \rm 
if for each box $\mathcal{B} \subseteq [0,1)^s$ 
which is a cartesian product of intervals contained 
in $[0,1)$,  of volume $\mid \!\!\mathcal{B}\!\mid$ we have
$$
\lim _{N \to \infty}{1\over N}\# \{ 1\leq n \leq N: x_n \in 
\mathcal{B} \} = \,\mid \!\!\mathcal{B}\!\mid.
$$
Equivalently, for continuous $f:[0,1)^s \to \Bbb{C}$ we have 
\begin{equation}
\label{unifdistribeq}
\lim _{N \to \infty}{1\over N}\sum _{n=1}^Nf(x_n)  
= \int _{[0,1)^s}f(t)dt.
\end{equation}
}
\textcolor{black}{
In the present context,
the left-hand side of
\eqref{kroneckeraverage} is the analog of
the left-hand side of \eqref{unifdistribeq}.
The right-hand side of 
\eqref{kroneckeraverage} is deduced from
Kronecker's Average Value Theorem. 
This Theorem was announced by Kronecker
at the Academy of Sciences, Berlin, 
in 1880, without proof. 
It has been 
given a proof by Rosen in
\cite{rosen}.
}

\

\section{Factorization, large gaps and Conjecture ``B"}
\label{S4}

The factorization of integer lacunary 
polynomials
has received a lot of attention, eg
Prasolov's book \cite{prasolov},
Schinzel \cite{schinzel67}
\cite{schinzel69} 
\cite{schinzel78}
\cite{schinzel00},
or
Filaseta 
and his collaborators
\cite{dobrowolskifilasetavincent}
\cite{filaseta}
\cite{filasetafinchnicol}
\cite{filasetamatthews}.
Finding irreducibilty criteria
is an important 
\textcolor{black}{topic}, 
and many problems remain open.

Consider polynomials of the form
\begin{equation}
\label{gNdefinition}
g_{N}(x) = x^N c(x^{-1}) + d(x),
\end{equation}
where $c$ and $d$ are fixed polynomials 
in $\zb[x]$
with $c(0), d(0)\neq 0$.
We are interested in the 
irreductibility
of $g_N$ for large $N$.
Such polynomials have already
appeared in different contexts,
eg \cite{dobrowolskifilasetavincent}
\cite{harringtonvincentwhite}
\cite{filasetamatthews}
\cite{filasetafordkonyagin}
\cite{schinzel67}.
From Schinzel \cite{schinzel00}
\cite{schinzel69}
and following
Sawin, Shusterman and Stoll
\cite{sawinshustermanstoll}
we first recall the general statements of
Theorem \ref{schinzelthm1_2}
and Corollary \ref{coroll1_3},
\textcolor{black}{concerned by} the factorization
of $g_N$ for $N$ large.

The step after, 
which is important
to understand the factorization of 
any $f \in \mathcal{B}$ containing large gaps, 
is to recognize such $f$ as a $g_N$,
and
apply these general theorems to $f$.
Of course only the $f \in \mathcal{B}$
having large gaps are concerned
\textcolor{black}{by} 
these statements. \textcolor{black}{
It is expected that the polynomials
$f \in \mathcal{B}$ 
which possess a small 
gappiness, though
not concerned by these statements,
have the same 
factorization properties.}

What is the main conclusion at large gaps?
Corollary \ref{coroll1_3}
implies that any $f \in \mathcal{B}$
would have an irreducible non-cyclotomic
component. 
But, in view of 
Theorem \ref{thm1factorization}, 
since the non-cyclotomic
part of  
$f(x)$ is $B(x)C(x)$ and that
$C(x)$ always
exist and is irreducible,
 it means that
the reciprocal non-cyclotomic
component $B(x)$ is trivial.  
So to say, Conjecture B is true 
at large gaps on $\mathcal{B}$.

Proving Conjecture B on all
$f \in \mathcal{B}$
amounts to \textcolor{black}{checking} Conjecture B
when the gaps in $f$ do not obey
the conditions of
Theorem \ref{schinzelthm1_2}
and Corollary \ref{coroll1_3}, 
that is at small gappiness.
Examples of pentanomials
with moderate gappiness
are given in Table \ref{table1}. 

Let us now state
Theorem \ref{schinzelthm1_2}
and Corollary \ref{coroll1_3},
and make precise the critical bounds
$N_1 , N_2 , N_3 , N_4$.

For a polynomial $u = \sum_{i=0}^{r}
a_i x^i \in \zb[x]$, we define 
$\|u\|$ as the squared Euclidean length  of its coefficient vector:
$$\|u\| := \sum_{i=0}^{r} |a_i|^2 ,$$
and $u^*$ denotes its reciprocal polynomial, as
$$u^{*}(x) = x^{\deg u} u(1/x) =
\sum_{i=0}^{r}
a_{r-i} x^i
.$$
We say that $u$ is reciprocal when
$u^* = u$. It is said to be non-reciprocal
if $u^* \neq u$.

From \eqref{gNdefinition} the
two integers $N_1$ and $N_2$ can be defined:
let $T:= \max\{\deg c , \deg d\}$,
and $\tau$ be the smallest Salem number known,
$1.176280$, unique
real positive root $> 1$ of
$x^{10}+x^9 - x^7 -x^6 - x^5 -x^4 -x^3 +x + 1$
(Lehmer's polynomial).
Denote
\begin{equation}
\label{N1index}
N_1 := \deg c + \deg d +
\left\{
\begin{array}{ll}
\frac{2 T}{\log \tau} \log (\|c\| + \|d\|)
& \mbox{if}~ T \leq 27,
\\
T (\log (6 T))^3 \log (\|c\| + \|d\|)
& \mbox{otherwise};
\end{array}
\right.
\end{equation}
$$
N_2 := \deg c + \exp\bigl(\frac{5}{16} \cdot
2^{(\|c\|+\|d\|)^2}\bigr)
\left(
2 + \max\{2, (\deg c)^2,
(\deg d)^2\}
\right)^{\|c\|+\|d\|}
.$$

\begin{definition}
A pair $(c,d)$ of polynomials
$c, d \in \zb[x]$ 
with 
$c(0), d(0) \neq 0$
is Capellian when $-d(x)/c(x^{-1})$ is a $p$th
power in $\qb(x)$ for some prime $p$
or
$d(x)/c(x^{-1})$ is 4 times
a fourth power in $\qb(x)$.
\end{definition}

The following Theorem 
\ref{schinzelthm1_2}
is Theorem 1.2 in
\cite{sawinshustermanstoll}, a revisited 
formulation of a theorem of Schinzel
(Theorem 74 in \cite{schinzel00});
the upper bound $N_2$ in 
item (ii) can be found in \cite{schinzel69}.

\begin{theorem}[Schinzel]
\label{schinzelthm1_2}
Let $c, d \in \zb[x]$
with $c(0), d(0) \neq 0$.
Assume that $c \neq \pm d$,
that
\textcolor{black}
{$\gcd_{\zb[x]} 
(c^*,d) =1$} 
and 
$(c,d)$ is not Capellian.
Then \begin{enumerate}
\item[(i)] there is a bound $N_0$ depending only 
on $c$ and $d$ such that for
$N > N_0$, the non-reciprocal part of $g_N$ is irreducible,
\item[(ii)] the bound $N_0$ satisfies $N_0 \leq N_2$.
\end{enumerate}
\end{theorem}

\begin{remark}
Theorem \ref{schinzelthm1_2} 
is probably the best
theorem on the subject in general.
However, in the case of the class 
$\mathcal{B}$, a direct
application of Ljunggren's tricks
provide a better solution, which is
Theorem \ref{thm1factorization},
\textcolor{black}{obtained recently by the authors
in \cite{dutykhvergergaugry18}}.
\end{remark}

Filaseta, Ford and Konyagin 
\cite{filasetafordkonyagin} have shown that the upper bound $N_2$ in (ii) can be replaced by
the better upper bound
$$N_3 :=
\deg c + 2 \max\{5^{4(\|c\|+\|d\|+t)-15},
T (5^{2(\|c\|+\|d\|+t)-8} + 1/4)
\}
$$
where \textcolor{black}{$T= \max\{\deg c , \deg d\}$
and}
$t$ is the number of terms in $c$ plus the
number of terms in $d$.

Under some assumptions, Sawin, Shusterman and Stoll
\cite{sawinshustermanstoll} show
that the upper bound $N_2$ in (ii)
can still be improved by
replacing $N_2$ by
$$N_4:= (1+\deg c + \deg d) 2^{\|c\|+\|d\|}.
$$
The bound $N_4$ is considerably smaller than
$N_3$, and the bound $N_2$ given by Schinzel
is extremely large. For instance,
for the pentanomial 
$f(x) = -1 +x+x^5+x^{14} + x^{100}$, 
the bounds are:
$N_1= 292$
\textcolor{black}{(cf Corollary \ref{coroll1_3} 
for its use)}, 
$N_2 = 1.54 \cdot 10^{4553919} , 
N_3 = 6 \cdot 10^{17}$
whereas $$N_4 = 480.$$
The authors in  \cite{sawinshustermanstoll}
also suggest an algorithm to
improve further the value $N_4$, by replacing
the exponential $2^{\|c\|+\|d\|}$ by 
a polynomial function of
$\|c\|+\|d\|$. This algorithm is useful in some 
cases.

\textcolor{black}{The integer $N_1$ defined in
\eqref{N1index} is introduced in
\cite{sawinshustermanstoll}
and both $N_1$ and
$N_2$
are used  when applying numerically
the following statement
(which can be found in \cite{sawinshustermanstoll}).
}

\begin{corollary}
\label{coroll1_3}
Under the assumptions of Theorem \ref{schinzelthm1_2},
if $N > \max\{N_0 , N_1\}$, then 
the non-cyclotomic part of $g_N$ is irreducible.
\end{corollary}

Now 
let 
$$f(x) = -1 + x + x^n + x^{m_1} +\ldots
+ x^{m_{j-1}}+ x^{m_{j}}+\ldots
+ x^{m_s} \quad \in \mathcal{B},$$
with $n \geq 3, s \geq 1$. It is easy
to write it under the general form $g_N$ 
as above.

Define $m_0 =n$ for coherency. 
With the distanciation rules we have 
missing monomials, those between
$x^n = x^{m_0}$ and $x^{m_1}$, between
$x^{m_1}$ and $x^{m_2}$, ..., and between
$x^{m_{s-1}}$ and $x^{m_s}$. Let us fix
an integer $j \in \{1, 2, \ldots, s\}$
and write $f$ as
\begin{equation}
\label{fgN}
f(x) = d(x) + x^{m_j} x^{m_s - m_j} c(x^{-1})
\end{equation}
with
$$
d(x) = -1 + x +x^n + x^{m_1}
+ \ldots +
x^{m_{j-1}}$$
and
$$
c(x) = x^{m_{s}-m_j} +  x^{m_s - m_{j+1}} +
x^{m_s - m_{j+2}}+
\ldots
+ x^{m_{s}-m_{s-1}} + 1.$$
The two polynomials $c$ and $d$ are fixed.
Let us make the link with
\eqref{gNdefinition}.
We consider the family of polynomials 
deduced from $f$ by taking arbitrarily sizes of the ``hole"
left between the monomials 
$x^{m_{j-1}}$ and $x^{m_j}$, as follows.
We consider the subcollection 
$$\widetilde{f^{(j)}} = 
\{f_{N}(x) := d(x)  + x^N c(x^{-1}) \mid
N \geq m_s \}
\quad \subset \mathcal{B}_n,$$
and are interested in the irreducibility
of $f_{N}$ for large $N$.
Let us note that $f_{m_s} = f$.
The superscript $\mbox{}^{(j)}$ means
that this family is ``associated" with 
$f$ and its
$j$th ``hole".
Let us observe that
the integer $\|c\| + \|d\|$ is the number of monomials
of any $f_N$ in
$\widetilde{f^{(j)}}$.
It is an \textcolor{black}{invariant of the family}:
\textcolor{black}{for any $f_N$ in
$\widetilde{f^{(j)}}$,
the integer $N_3$ is the same, and
is equal to}
$$N_3 :=
\deg c + 2 \max\{5^{8(\|c\|+\|d\|)-15},
T (5^{4(\|c\|+\|d\|)-8} + 1/4)
\}.
$$
\vspace{0.3cm}

\noindent
{\it Example.--} We consider infinite 
families of
pentanomials
which present a variable gappiness at the
last monomial. The polynomial $c$ is taken equal to 1.
When $c=1$, then assumptions of Theorem
\ref{schinzelthm1_2}
are satisfied.
The pentanomials are defined below. 
When applying  
Corollary \ref{coroll1_3} to
the polynomials
$f_N \in \widetilde{f^{(j)}}$ for $j=m_2$
and
$N > \max\{N_4 , N_1\}$, then
Conjecture B is valid for
all the $f_N \in \widetilde{f^{(j)}}$.
In Table \ref{table1} 
we check the validity of
Conjecture B
for the intermediate values of $N$ in the range
$m_1 + (n-1) \leq N \leq N_4$.

\begin{theorem}[Finch - Jones]
\label{thmfinchjones}
Let $d \in \mathcal{B}$,
$$d(x) = -1 + x + x^n + x^{m_1} .$$
Let $e_1 = {\rm gcd}(m_1 , n-1),
e_2 = {\rm gcd}(n, m_1 -1)$.
The quadrinomial $d(x)$ is irreducible over
$\qb$ if and only if
$$m_1 \not\equiv 0~({\rm mod}~ 2 \,e_1) ,
\qquad 
n \not \equiv 0~({\rm mod}~ 2 \,e_2) .$$
\end{theorem}

In Table \ref{table1}
the Conjecture B is tested on all the
pentanomials
$f(x) = d(x) + x^N c(x^{-1})$
with
$d(x)=-1+x+x^n + x^{m_1}$,
$c(x^{-1}) = 1$, and for
$N$ in the range 
$\{m_1 + n-1 , m_1 + n, \ldots, N_4\}$. 
The quadrinomials
$d(x)$ are chosen to be 
either irreducible 
(in which case they are labelled ``+")
or reducible
(in which case they are labelled ``-")
after Finch-Jones's Theorem 
\ref{thmfinchjones}.
The pentanomials
$f(x) = 1+x+x^n + x^{m_1} + x^N$
obtained are either irreducible, or
have cyclotomic factors
$\Phi_{k}(x)$ with 
$k=3, 6, 9, 10, 12, $
$18, 24$ or $30$. 
No reciprocal non-cyclotomic factor
appears in the factorizations.

\begin{table}[H]
\centering
\caption{\small\em Numerical verification of the Conjecture B for all 
low degree pentanomials 
$f(x) = -1+x+x^n + x^{m_1} + x^N$
for which $m_1 + (n-1) \leq N \leq N_4$.
For each experiment number,
the polynomials $f(x)$ are either irreducible or
have cyclotomic factors in 
$\{\Phi_{k}\}$
(e.g. for Exp. num. equal to 1, only
$\Phi_{k=3}$ and $\Phi_{k=6}$ 
are encountered).}

\label{table1}
\bigskip
{\small \begin{tabular}{ccccccc}
    \hline
   Exp. num. & $n$ & $m_{\,1}$ & Quad. irred. & $N_{\,4}$ & $\{\Phi_{\,k}\,(\,x\,)\}$ & $B-$conj. \\
     \hline
    $1$ & $3$ & $5$ & $+$ & $192$ & $\{3,\,6\}$ & $\checkmark$ \\
    $2$ & $3$ & $6$ & $+$ & $224$ & $\{3,\,6\}$ & $\checkmark$ \\
    $3$ & $3$ & $7$ & $+$ & $256$ & $\{10\}$ & $\checkmark$ \\
    $4$ & $3$ & $8$ & $-$ & $288$ & $\{3\}$ & $\checkmark$ \\
    $5$ & $3$ & $9$ & $+$ & $320$ & $\{3,\,10\}$ & $\checkmark$ \\
    $6$ & $3$ & $10$ & $+$ & $352$ & $\{18\}$ & $\checkmark$ \\
    $7$ & $3$ & $11$ & $+$ & $384$ & $\{3,\,6\}$ & $\checkmark$ \\
    $8$ & $3$ & $12$ & $-$ & $416$ & $\{3,\,6\}$ & $\checkmark$ \\
    $9$ & $3$ & $13$ & $+$ & $448$ & $\emptyset$ & $\checkmark$ \\
    $10$ & $3$ & $17$ & $+$ & $576$ & $\{3,\,6,\,10\}$ & $\checkmark$ \\
    $11$ & $4$ & $7$ & $-$ & $256$ & $\{6,\,9\}$ & $\checkmark$ \\
    $12$ & $4$ & $8$ & $-$ & $288$ & $\emptyset$ & $\checkmark$ \\
    $13$ & $4$ & $9$ & $+$ & $320$ & $\{9\}$ & $\checkmark$ \\
    $14$ & $4$ & $10$ & $-$ & $352$ & $\emptyset$ & $\checkmark$ \\
    $15$ & $4$ & $11$ & $-$ & $384$ & $\{6\}$ & $\checkmark$ \\
    $16$ & $4$ & $12$ & $-$ & $416$ & $\emptyset$ & $\checkmark$ \\
    $17$ & $4$ & $13$ & $+$ & $448$ & $\{6,\,24\}$ & $\checkmark$ \\
    $18$ & $4$ & $17$ & $+$ & $576$ & $\{6\}$ & $\checkmark$ \\
    $19$ & $5$ & $9$ & $+$ & $320$ & $\{3,\,6,\,12\}$ & $\checkmark$ \\
    $20$ & $5$ & $10$ & $+$ & $352$ & $\{6\}$ & $\checkmark$ \\
    $21$ & $5$ & $11$ & $+$ & $384$ & $\{6\}$ & $\checkmark$ \\
    $22$ & $5$ & $12$ & $+$ & $416$ & $\{3,\,6,\,12\}$ & $\checkmark$ \\
    $23$ & $5$ & $13$ & $+$ & $448$ & $\{6\}$ & $\checkmark$ \\
    $24$ & $5$ & $14$ & $+$ & $480$ & $\{6\}$ & $\checkmark$ \\
    $25$ & $5$ & $15$ & $+$ & $512$ & $\{3,\,6\}$ & $\checkmark$ \\
    $26$ & $5$ & $16$ & $-$ & $544$ & $\{6,\,30\}$ & $\checkmark$ \\
    $27$ & $5$ & $17$ & $+$ & $576$ & $\{6\}$ & $\checkmark$ \\
     \hline
    
  \end{tabular}
  }
\end{table}

\begin{remark}
The assumptions of Theorem
\ref{schinzelthm1_2} are not strong
and are compatible with Conjecture B.
Indeed,
if $f(x) = d(x) + x^N c^{*}(x)$ 
belongs to
$\mathcal{B}$, then
$c \neq \pm d$ always, the couple
$(c, d)$ is never Capellian. 
What about the assumption
gcd$_{\zb[x]}(c^*, d)=1$?
The polynomial $c^{*}(x)$ is 
a Newman polynomial
since all the coefficients are in $\{0, 1\}$.
By the Odlyzko-Poonen Conjecture 
\cite{breuillardvarju} 
it is irreducible with
probability 1.
But it has no zero in the interval $[0,1]$,
which is not the case of the non-reciprocal part
of $d(x) \in \mathcal{B}$ by Theorem \ref{thm1factorization}.
Therefore, with probability one,
$c^*$ cannot be the non-reciprocal part
of $d$. Therefore, with
probability one, it
is an irreducible cyclotomic polynomial, 
which is
a cyclotomic factor of $d$.
At \textcolor{black}{worst}, gcd$_{\zb[x]}(c^*, d)$ would be a cyclotomic factor of $f$ with probability one.
\end{remark}

Applying Corollary \ref{coroll1_3},
the numerical investigation reported in Table
\ref{table2} allows to complement completely 
the study of Conjecture B at large gaps by
the one at intermediate lacunarity,
and gives a proof to the following result.

\begin{proposition}
\label{}
The family of pentanomials
$f(x) = -1 +x+x^5+x^{14} + x^{N}$,
$N \geq 18$, of $\mathcal{B}_5$, admits
the bounds
$N_1= 292$, $N_4 = 480$.
All the polynomials of this family
satisfy Conjecture B.
\end{proposition}

\begin{remark}
The other infinite families of pentanomials 
whose first quadrinomial is given in the list 
of 
\textcolor{black}
{Table \ref{table1}}
present different values of $N_1$ and $N_4$,
and may be studied in the same way, 
with respect to Conjecture B
\end{remark}

In the continuation of 
\textcolor{black}
{search for} the conditions
of existence of very small Mahler measures 
M$(\beta) > 1$ of
reciprocal algebraic integers $\beta > 1$, 
close to 1,
it should be noticed that the gaps
of the Parry Upper functions
$$f_{\beta}(x) := -1 + x + x^n +
x^{m_1} + x^{m_2} + \ldots + x^{m_s} +
\ldots$$
where $m_1 - n \geq n-1$, 
$m_{q+1}-m_q \geq n-1$ for $q \geq 1$,
are never large by the following asymptotic
upper bound \cite{vergergaugry06}:
$$\limsup_{j \to \infty} 
\frac{m_{j+1}}{m_j} \leq
\frac{\log {\rm M}(\beta)}{\log \beta}.$$
Consequently the polynomial sections
of $f_{\beta}(z)$ have asymptotically a moderate 
gappiness, which is the intermediate 
domain of study for the non-existence
of reciprocal non-cyclotomic factors.
If the domain of very large gaps is covered by
Corollary \ref{coroll1_3}, 
the existence of non-zero
reciprocal algebraic integers $\beta > 1$
would lead to the difficult domain of
intermediate gappinesses.

\section{On the lower bound in Serre's density theorem}
\label{S5}

In this paragraph we 
show on examples 
that 
the lower bound given by Serre
in Theorem \ref{densitythm} (i) of the density 
of the set
$\mathcal{P}_0$ is 
\textcolor{black}{far from being sharp for
the 
polynomials of the class $\mathcal{B}$}.

We consider the set of
the polynomial sections arising from 
$\zeta_{\tau}(z)$ with $\tau = 1.176280$ Lehmer's number.

Let $f_{\tau}(x) = -1+x+x^{12}+x^{31}
+ x^{44} + x^{63} + x^{86} + x^{105}
+ x^{118}+\ldots = -1/\zeta_{\tau}(x)$
be the Parry Upper function
at Lehmer's number $\tau = 1.176280\ldots$. 
All the polynomial sections belong to the 
subclass
$\mathcal{B}_{12}$. Denote

$S_0 = -1+x+x^{12}$, irreducible,

$S_1 = -1+x+x^{12}+x^{31}$, reducible,
$ = (x^2 + 1)(x^4 - x^2 + 1) C_{1}(x)$,

$S_2 = -1+x+x^{12}+x^{31} + x^{44}$, irreducible,

$S_3 = -1+x+x^{12}+x^{31} + x^{44} + x^{63}$,
irreducible,

$S_4 = -1+x+x^{12}+x^{31}
+ x^{44} + x^{63} + x^{86}$, irreducible,

$S_5 = -1+x+x^{12}+x^{31}
+ x^{44} + x^{63} + x^{86} + x^{105}$,
reducible, $= \Phi_{3}(x) \Phi_{4}(x) \Phi_{12}(x) C_{5}(x)$,

$S_6 = -1+x+x^{12}+x^{31}
+ x^{44} + x^{63} + x^{86} + x^{105}
+ x^{118}$, irreducible.

The numbers of solutions
of the non-reciprocal parts of
$S_j$, 
$C_{j}(x) \equiv 0$ mod $p$
with $p \leq 43$,
are given
in Table \ref{table2}, for $0 \leq j \leq 6$.
In this \textcolor{black}{table}, on each line, 
the frequencies of zeroes are
substantially higher than the ones deduced
from the general
bounds $1/\deg (C_j)$ of
Theorem \ref{densitythm} (i). 

\begin{table}[H]
\centering
\caption{\small\em Values of 
the quantities $N_{p}(S_j)$ for all
primes $p$ in the range $\{2, 3, \ldots, 43\}$,
where $S_j$ is the $j$th polynomial section of the
Parry Upper function 
$f_{\tau}(x) = -1/\zeta_{\tau}(x)$,
and $\tau$ Lehmer's number, if $S_j$ is irreducible. When $S_j$ is not
irreducible the quantity $N_{p}(S_j)$
represented
is replaced by  $N_{p}(C_j)$ where $C_j$ is the non-reciprocal part of $S_j$.
The contributions of the cyclotomic parts
modulo $p$, removed from the lines $j=1$ and $j=5$ are indicated underneath (calculated with PARI/GP).
}
\label{table2}
\bigskip
{\small
\begin{tabular}{ccccccccccccccc}
    \hline
$j$ &  $2$ & $3$ & $5$ & $7$ & $11$ & $13$ & $17$ & $19$ & $23$ & $27$ & $31$ & $37$ & $41$ & $43$\\
     \hline
$0$ & $0$ & $0$ & $0$ & $0$ & $2$ & $0$ & 
$2$ & $1$ & $1$ & $1$ & $0$ & $1$ & $0$ & $1$\\
$1$ & $0$ & $0$ & $0$ & $0$ & $0$ & $0$ & 
$1$ & $1$ & $1$ & $1$ & $0$ & $1$ & $2$ & $1$\\
$2$ & $0$ & $1$ & $0$ & $1$ & $0$ & $0$ & 
$1$ & $1$ & $0$ & $0$ & $2$ & $2$ & $2$ & $0$\\
$3$ & $1$ & $0$ & $0$ & $1$ & $0$ & $0$ & 
$1$ & $1$ & $1$ & $1$ & $1$ & $0$ & $1$ & $1$\\
$4$ & $0$ & $0$ & $1$ & $0$ & $0$ & $1$ & 
$2$ & $2$ & $0$ & $1$ & $0$ & $3$ & $0$ & $1$\\
$5$ & $0$ & $0$ & $0$ & $0$ & $0$ & $1$ & 
$1$ & $0$ & $0$ & $2$ & $1$ & $0$ & $0$ & $0$\\
$6$ & $0$ & $0$ & $0$ & $1$ & $0$ & $1$ & 
$1$ & $1$ & $1$ & $0$ & $1$ & $0$ & $3$ & $1$\\
\hline
\hline
$N_{p}(X^2 +X +1)$ & $0$ & $1$ & $0$ & $2$ & $0$ & $2$ & 
$0$ & $2$ & $0$ & $0$ & $2$ & $2$ & $0$ & $2$\\
$N_{p}(X^2 +1)$ & $1$ & $0$ & $2$ & $0$ & $0$ & $2$ & 
$2$ & $0$ & $0$ & $2$ & $0$ & $2$ & $2$ & $0$\\
$N_{p}(X^4 - X^2 +1)$ & $0$ & $0$ & $0$ & $0$ & $0$ & $4$ & 
$0$ & $0$ & $0$ & $0$ & $0$ & $4$ & $0$ & $0$\\

     \hline
  \end{tabular}
}
\end{table}

\section{Counting solutions mod $p$ and letting $p$ tend to infinity} 
\label{S6}

For understanding
the quantities $N_{p}(f)$,
$f \in \mathcal{B}$, and possibly
relate them to the coefficients
of some power series,
by the global correspondence Langlands program,
we follow the general strategy of Serre
in \cite{serre} \cite{serreVID1}
\cite{serreVID2}. The starting point is
the study of
the roots of the trinomials
$-1+x+x^n$ mod $p$, then of the
quadrinomials of $\mathcal{B}$ mod $p$,
whose first three terms
are $-1+x+x^n$, etc, then ideally all 
the elements of the class $\mathcal{B}$.

On one side the numbers
$N_{p}(f)$, $f \in \mathcal{B}$,
are correlated to the factorization
of the polynomials $f$ via
Kronecker's Average Value Theorem 
\ref{mainN}. On the other side the
numbers $N_{p}(f)$ are related 
to questions of modularity 
\textcolor{black}{
and geometry. 
}

\subsection{Trinomials $-1+z+z^n$ mod $p$ and Newforms}
\label{S6.1}

The case $n=2$:
$f(x)=-1+x+x^2$. 
The discriminant of $f$ is 5.
The polynomial $f$ has a double root mod 5;
then $N_{5}(f) =1$. 
If $p \neq 2, 5$,
the roots of $f$ in
$\overline{\fb_p}$ are
$(1\pm \sqrt{5})/2$.
If $5$ is a square mod $p$ then
there are two roots, $N_{p}(f)=2$.
If not $N_{p}(f)=0$.

For $p$ and $q$ two distinct odd
prime numbers, define the Legendre symbol as

$$\left(\frac{q}{p}\right)
=
\left\{
\begin{array}{ll}
+1 & \mbox{if $r^2 \equiv q$ mod $p$ for some integer $r$,}\\
-1 & \mbox{otherwise}.
\end{array}
\right.$$
The law of quadratic reciprocity says
$$\left(\frac{q}{p}\right)
\left(\frac{p}{q}\right)
=
(-1)^{\frac{p-1}{2} \frac{q-1}{2}}.$$
Then 5 is a square mod p if and only if
$p \equiv \pm 1$ mod $5$. 
Therefore
$$N_{p}(-1+x+x^2)
=
\left\{
\begin{array}{cc}
0 & \mbox{if } p \equiv \pm 2 \mod 5,\\
2 & \mbox{if } p \equiv \pm 1 \mod 5.
\end{array}
\right.$$
We deduce
the distribution of values of 
$N_{p}(-1+x+x^2)$ in the first column of
Table \ref{trinomialsmodp}.
This distribution seems
fairly regular. The probability limit for each value
$0$ and $2$ is 1/2, but there is a Chebyshev bias
\cite{rubinsteinsarnak}, mentioned
in Serre \cite{serre}, which slightly 
shifts the 
probability distribution to $0$ preferentially.
It is observed in Table \ref{chebyshevbias}.
This bias occurs for polynomials
$f \in \mathcal{B}$ and will be studied elsewhere.

Using the change of variable $x$ to $-x$,
the results of \cite{serre} section 5.2, can
be directly applied to the trinomials
$-1+x+x^2$, as follows.

Let us consider the $q$-series
$F=\sum_{m =0}^{\infty} a_m q^m$, defined
by
$$F=\frac{q-q^2 -q^3 +q^4}{1-q^5} =
q-q^2 -q^3 +q^4+q^6 -q^7 -q^8 +q^9+\ldots$$
Then
$$N_{p}(-1+x+x^2) = a_p + 1
\qquad \mbox{for all prime numbers}~p.$$
Note that the coefficients $(a_m)$ of $F$ have the property to be strongly
multiplicative, in the sense:
$$a_{r q} = a_{r} a_{q}\qquad \mbox{for all integers}~
r, q \geq 1.$$ 
The corresponding Dirichlet series 
\cite{LMFDB}
is the $L$-series
$$\sum_{m =1}^{\infty} \frac{a_m}{m^s}
=
\prod_p \left(
1-\left(\frac{p}{5}\right) p^{-s}
\right)^{-1}.
$$

\begin{table}[H]
\centering
\caption{\small\em Values of 
the numbers $N_{p}(-1+x+x^n)$ for all
primes $p$ in the range $\{2, 3, \ldots, 101\}$
and $n$ in the range $\{2, 3, \ldots, 15\}$
(calculated with PARI/GP).
}
\label{trinomialsmodp}
\bigskip
\begin{tabular}{ccccccccccccccc}
    \hline
$p$ &  $n=2$ & $3$ & $4$ & $5$ & $6$ & $7$ & $8$ & $9$ & $10$ & $11$ & $12$ & $13$ & $14$ & $15$\\
    \hline
$2$ & $0$ & $0$ & $0$ & $0$ & $0$ & $0$ & 
$0$ & $0$ & $0$ & $0$ & $0$ & $0$ & $0$ & $0$\\

$3$ & $0$ & $1$ & $0$ & $1$ & $0$ & $1$ & 
$0$ & $1$ & $0$ & $1$ & $0$ & $1$ & $0$ & $1$\\

$5$ & $1$ & $0$ & $0$ & $1$ & $1$ & $0$ & 
$0$ & $1$ & $1$ & $0$ & $0$ & $1$ & $1$ & $0$\\

$7$ & $0$ & $0$ & $1$ & $2$ & $0$ & $1$ & 
$0$ & $0$ & $1$ & $2$ & $0$ & $1$ & $0$ & $0$\\

$11$ & $2$ & $1$ & $1$ & $1$ & $0$ & $1$ & 
$1$ & $0$ & $0$ & $1$ & $2$ & $1$ & $1$ & $1$\\

$13$ & $0$ & $1$ & $1$ & $2$ & $1$ & $0$ & 
$0$ & $0$ & $0$ & $2$ & $0$ & $1$ & $0$ & $1$\\

$17$ & $0$ & $1$ & $2$ & $1$ & $1$ & $0$ & 
$0$ & $1$ & $0$ & $1$ & $2$ & $1$ & $1$ & $0$\\

$19$ & $2$ & $0$ & $0$ & $3$ & $1$ & $0$ & 
$1$ & $1$ & $0$ & $3$ & $1$ & $0$ & $1$ & $1$\\

$23$ & $0$ & $1$ & $1$ & $0$ & $0$ & $2$ & 
$2$ & $2$ & $0$ & $0$ & $1$ & $0$ & $1$ & $1$\\

$29$ & $2$ & $1$ & $1$ & $1$ & $2$ & $0$ & 
$0$ & $2$ & $0$ & $1$ & $1$ & $0$ & $1$ & $0$\\

$31$ & $2$ & $2$ & $0$ & $0$ & $1$ & $1$ & 
$1$ & $1$ & $1$ & $3$ & $0$ & $1$ & $1$ & $0$\\

$37$ & $0$ & $1$ & $2$ & $0$ & $1$ & $0$ & 
$0$ & $0$ & $0$ & $2$ & $1$ & $0$ & $1$ & $1$\\

$41$ & $2$ & $0$ & $1$ & $1$ & $0$ & $0$ & 
$0$ & $0$ & $1$ & $0$ & $0$ & $0$ & $0$ & $1$\\

$43$ & $0$ & $1$ & $0$ & $1$ & $1$ & $2$ & 
$0$ & $1$ & $0$ & $2$ & $1$ & $0$ & $0$ & $1$\\

$47$ & $0$ & $3$ & $0$ & $0$ & $2$ & $2$ & 
$1$ & $0$ & $2$ & $2$ & $0$ & $2$ & $0$ & $2$\\

$53$ & $0$ & $1$ & $2$ & $0$ & $1$ & $0$ & 
$0$ & $1$ & $0$ & $2$ & $0$ & $0$ & $0$ & $0$\\

$59$ & $2$ & $0$ & $1$ & $1$ & $1$ & $2$ & 
$1$ & $0$ & $2$ & $1$ & $0$ & $0$ & $2$ & $1$\\

$61$ & $2$ & $1$ & $1$ & $1$ & $1$ & $1$ & 
$1$ & $1$ & $0$ & $2$ & $0$ & $1$ & $0$ & $1$\\

$67$ & $0$ & $3$ & $2$ & $0$ & $1$ & $0$ & 
$0$ & $0$ & $0$ & $3$ & $0$ & $1$ & $1$ & $1$\\

$71$ & $2$ & $0$ & $0$ & $0$ & $0$ & $0$ & 
$1$ & $1$ & $1$ & $1$ & $0$ & $1$ & $2$ & $1$\\

$73$ & $0$ & $1$ & $0$ & $1$ & $0$ & $0$ & 
$0$ & $0$ & $2$ & $3$ & $0$ & $0$ & $1$ & $0$\\

$79$ & $2$ & $1$ & $2$ & $0$ & $2$ & $1$ & 
$1$ & $0$ & $2$ & $2$ & $1$ & $1$ & $0$ & $0$\\

$83$ & $0$ & $1$ & $4$ & $0$ & $3$ & $3$ & 
$2$ & $1$ & $0$ & $1$ & $0$ & $0$ & $2$ & $1$\\

$89$ & $2$ & $1$ & $1$ & $1$ & $0$ & $0$ & 
$0$ & $1$ & $0$ & $1$ & $1$ & $1$ & $2$ & $0$\\

$97$ & $0$ & $0$ & $1$ & $0$ & $1$ & $1$ & 
$2$ & $0$ & $0$ & $3$ & $1$ & $1$ & $1$ & $0$\\

$101$ & $2$ & $0$ & $0$ & $1$ & $0$ & $0$ & 
$0$ & $0$ & $0$ & $1$ & $1$ & $1$ & $0$ & $0$\\
\hline
\hline

     \hline
  \end{tabular}
\end{table}

\begin{table}[H]
\centering
\caption{\small\em Chebyshev bias for
$-1+x+x^2$ mod $p$.}
\label{chebyshevbias}
\bigskip
{\small \begin{tabular}{cccc}
    \hline
    & $x = 101$ & $1001$ & $10001$  \\
     \hline
    $\#\{ p \leq x\mid N_{p}(-1+x+x^2) = 0\}$ 
    & $14$ & $89$ & $619$  \\
    
    $\#\{ p \leq x \mid N_{p}(-1+x+x^2) = 2\}$ 
    & $11$ & $78$ & $609$  \\
    
    $\#\{ p \leq x\}$ & $26$ & $168$ & $1229$ \\
     \hline   
    
  \end{tabular}
  }
\end{table}
\vspace{0.3cm}

The case $n=3$:
the discriminant of $f(x)=-1+x+x^3$ is -31.
Modulo 31, the polynomial $f$ has a double 
root and a simple root. Hence
$N_{31}(f) = 2$.
For $p \neq 31$, one 
has
\begin{equation}
\label{N31_01}
N_{p}(f)=
\left\{
\begin{array}{ll}
0 \mbox{~or~} 3 & \mbox{if} \left(\frac{p}{31}\right)=+1,\\
1 & \mbox{if} \left(\frac{p}{31}\right)=-1.
\end{array}
\right.
\end{equation}

This explains the values reported in
the corresponding column
in Table \ref{trinomialsmodp}.

\vspace{0.3cm}

The case $n=4$: using the change of variable 
$x$ to $-x$, the case of the trinomial
$-1+x+x^4$ may be deduced from the case 
of  $-1-x-x^4$. In Section 5.4
of \cite{serre} Serre gives
the values $N_{p}(-1-x+x^4)$ from coefficients
of Newforms.
Let us summarize the expressions he obtains,
for the trinomials $-1+x+x^4$.

The discriminant of $f(x) = -1+x+x^4$
is $-283$. Modulo $283$, $f$ has one double root 
and two simple roots. Then $N_{283}(f) = 3$.
If $p\neq 283$, one has
$$N_{p}(f) =
\left\{
\begin{array}{ll}
0 ~\mbox{or}~ 4 & \mbox{if $p$ can be written as
$x^2 + x y + 71 y^2$},\\
1 & \mbox{if $p$ can be written as
$7x^2 + 5 x y + 11 y^2$},\\
0 ~\mbox{or}~ 2 & \mbox{if}
\left(\frac{p}{283}
\right)=-1.
\end{array}
\right.
$$

A complete determination 
of $N_{p}(f)$ can be obtained
via a Newform
$F=\sum_{m = 0}^{\infty} a_m q^m$
of weight 1 and level 283 
\cite{LMFDB} whose first hundred 
terms are given in
Crespo \cite{crespo}:
$$F=q + i \sqrt{2} q^2 -i\sqrt{2} q^3 - q^4
-i \sqrt{2} q^5 +2 q^6 -q^7 -q^9
+2 q^{10} + q^{11}
i \sqrt{2} q^{12} +q^{13}$$
$$-i\sqrt{2}q^{14} -2 q^{15} -q^{16} -i\sqrt{2}q^{18}
+i\sqrt{2}q^{19} +i\sqrt{2}q^{20}
+i\sqrt{2}q^{21} +i\sqrt{2}q^{22}
-q^{23}$$
$$-q^{25} +i\sqrt{2}q^{26} +q^{28}-q^{29}
-2 i \sqrt{2} q^{30}
+i \sqrt{2} q^{31}
-i \sqrt{2} q^{32}
-i \sqrt{2} q^{33}
+i \sqrt{2} q^{35}
$$
$$
+q^{36} -2 q^{38}
-i \sqrt{2} q^{39}+q^{41}
-2 q^{42}
+\ldots$$

Then one has:
$$N_{p}(f) =1+ (a_{p})^2
- \left(\frac{p}{283}
\right)\qquad
\mbox{for all primes $p \neq 283$}.$$
This explains the values reported in
the corresponding column
in Table \ref{trinomialsmodp}.

\vspace{0.3cm}

The case $n=7$: the Galois group $G$ of the trinomial
$-1+x+x^7$ is $S_7$
\cite{cohenmovahedisalinier}. Applying 
Theorem \ref{densitythm} (ii) gives
$1.984 \cdot 10^{-4}
=
(\#G)^{-1}$ for the density of primes $p$
such that $N_{p}(-1+x+x^7)=7$.
Indeed, the only two prime numbers
$\leq 10^5$ realizing the maximality
identity $N_{p}(-1+x+x^7)=7$
 are:
$p=41143$ and
$p=82883$.
By Theorem \ref{densitythm} (i) the density
of prime numbers $p$ such that
$N_{p}(1+x+x^7)=0$ exists and is above $1/7$.
This is compatible with
the corresponding column
in Table \ref{trinomialsmodp}.

\vspace{0.3cm}

The general case $n$: following \cite{serre}
there should exist formulas
for the numbers $N_{p}(f)$
coming from the coefficients of
$q$-series, newforms, etc.

\subsection{Densities and lacunarity}
\label{S6.2}

Let $n \geq 3$
and
$f(x)= (-1+x+x^n ) + 
x^{m_1} +x^{m_2} +\ldots+ x^{m_s} \quad 
\in \mathcal{B}$. To align
the statements on
the presentation of Serre \cite{serreVID1}
\cite{serreVID2}, 
let us adopt the geometric language
of schemes. The system of algebraic 
equations defining the
variety is reduced to one equation.
Denote by $A=\zb[x]/(f)$ the
finitely generated ring over $\zb$. Let
$X= {\rm Spec}(A)$. The 
solutions mod $p$ of $f(x) \equiv 0$
correspond to the elements 
$x \in X(\fb_p)$ in the fiber.
We denote $N_{p}(f)$ by $N_{X}(p)$.
The polynomial
$f$ is fixed and $p$ varies.

\begin{theorem}
For any integer $q \geq 1$, 
any $\gamma \in \zb/q\zb$,
the
set
$$\overline{\mathcal{P}_{\gamma}}
:= \{p \mid N_{X}(p) \equiv \gamma \mod q\}$$
has a density. This density
is a rational number.
\end{theorem}
\begin{proof}
\cite{serreVID1} \cite{serreVID2}. 
\end{proof}
For $q=2$ the question of the density
of $\overline{\mathcal{P}_{\gamma}}$
amounts to understand when 
$N_{X}(p)$ is even, and when
$N_{X}(p)$ is odd.
In general the set
$\overline{\mathcal{P}_{\gamma}}$ is empty,
is a finite set or
has a density 
which is $> 0$. A case 
when the density is $> 0$
is of topological origin
and
comes from the topology
of the complex space $X(\cb)$,
resp. of the real space
$X(\rb)$.
Denote by
$\chi(X(\cb))$, resp $\chi_c(X(\rb))$, in
$\zb$, the Euler characteristic
of $X(\cb)$, resp. the Euler characteristic
with compact support of $X(\rb)$.

\begin{theorem}
For any integer $q \geq 1$, 
$$(i) \qquad\delta\Bigl(
\overline{\mathcal{P}_{\gamma}}
\Bigr) > 0
\qquad \quad \mbox{for}~~ \gamma = \chi(X(\cb)),$$
$$(ii)\qquad
\delta\Bigl(
\overline{\mathcal{P}_{\gamma}}
\Bigr) > 0
\qquad \quad \mbox{for}~~ \gamma = \chi_c(X(\rb)).$$
\end{theorem}
\begin{proof}
\cite{serreVID1} \cite{serreVID2}.
\end{proof}

The decomposition of $f$
as the sum of the trinomial part
$$-1+x+x^n$$
and the perturbation term
$$x^{m_1}+x^{m_2} + \ldots
+ x^{m_s}$$
for $m_1 - n \geq n-1 , m_{j+1} - m_j \geq n-1 \mbox{~for } 1 \leq j \leq s-1$ suggests to define
$Y={\rm Spec}(\zb[X]/(-1+x+x^n))$
and to study $N_{Y}(p)$
for $p$ tending to infinity.
A first question is about the comparison between
$N_{X}(p)$ and $N_{Y}(p)$, and when $p$ tends 
to infinity.

\textcolor{black}{\bf Question 1}: For any integer $q \geq 1$,
any $\gamma \in \zb/q\zb$ has 
$$\overline{\mathcal{P}_{\gamma}}
:= \{p \mid N_{X}(p) - N_{Y}(p) \equiv \gamma \mod q\}$$
a density?

\textcolor{black}{\bf Question 2}: 
\textcolor{black}{What} is the role of $n$ in the sets
$\{p \mid N_{X}(p) - N_{Y}(p) \equiv \gamma \mod q\}$,
in particular when $n$ becomes very large?

\
\textcolor{black}{
\section{Appendix: The expression of the dynamical zeta function $\zeta_{\beta}(z)$ of the $\beta$-shift when $\beta > 1$ is close to 1}
\label{appendixzeta}
}
\textcolor{black}{
The importance of the class of polynomials 
$\mathcal{B}$
comes from
the formulation of the dynamical zeta function
$\zeta_{\beta}(z)$ of the $\beta$-shift when 
$\beta > 1$ is close to 1. Let us recall the main 
steps, leaving the details for the reader.
}

\textcolor{black}{
The notion of dynamical zeta function
was introduced by M. Artin and B. Mazur
\cite{artinmazur} in 1965.
Let $h : V \mapsto V$ be a diffeomorphism of a compact
manifold $V$, such that its iterates $h^k$ all
have isolated fixed points. Then they
defined
\begin{equation}
\label{dynamicalfunctiongeneral}
\zeta_{\beta}(z) := \exp\Bigl(
\sum_{n=1}^{\infty} \, 
\frac{\#\{x \in V \mid 
h^{n}(x) = x\}}{n} \, z^n
\Bigr),
\end{equation}
and showed that for a dense set of diffeomorphisms
$h$ of $V$ the power series of such an expression
converged in a neighbourhood of $z=0$.
The dynamical zeta function of a dynamical system, when
defined,
is an analytic function
which concentrates a lot of information
on the dynamical system, and therefore is a powerful tool
(Pollicott \cite{pollicott}).  
}

\textcolor{black}{
From Theorem 2 in \cite{baladikeller}, concerning the 
R\'enyi-Parry dynamical numeration system 
$(V=[0,1], h=T_{\beta})$, where $\beta > 1$ and
$T_{\beta}: x \to \beta x \mod 1$ is the $\beta$-transformation
\cite{itotakahashi}
\cite{lagarias}, we deduce
}
\textcolor{black}{
\begin{theorem}
\label{dynamicalzetatransferoperator}
Let $\beta \in (1,\theta_{2}^{-1})$. Then, 
the Artin-Mazur
dynamical zeta function
\begin{equation}
\label{dynamicalfunction}
\zeta_{\beta}(z) := \exp\Bigl(
\sum_{n=1}^{\infty} \, 
\frac{\#\{x \in [0,1] \mid 
T_{\beta}^{n}(x) = x\}}{n} \, z^n
\Bigr),
\end{equation}
counting the number of periodic 
points of period dividing $n$,
is nonzero and meromorphic
in $\{z \in \cb : |z| < 1\}$, and such that
$1/\zeta_{\beta}(z)$ is holomorphic
in $\{z \in \cb : |z| < 1\}$,
\end{theorem}
}

\textcolor{black}{
Let $\beta > 1$ be a real number. Denote
$\mathcal{A}_{\beta} := \{0, 1, 2, \ldots, 
\lceil \beta - 1 \rceil \}$, where
$\lceil \beta - 1 \rceil$ denotes the upper integer part
of $\beta - 1$. If $\beta$ is not
an integer, then 
$\lceil \beta - 1 \rceil
= \lfloor \beta \rfloor$ which is the usual 
integer part of $\beta$.
Using ergodic theory, 
Takahashi \cite{takahashi}
and Ito Takahashi \cite{itotakahashi}
obtained the reformulation
of \eqref{dynamicalfunction}
as follows.
}
\textcolor{black}{
\begin{theorem}
\label{parryupperdynamicalzeta}
Let $\beta > 1$ be a real number.
Then 
\begin{equation}
\label{zetafunctionfraction00}
\zeta_{\beta}(z) = \frac{1 - z^N}{(1 - \beta z)
\Bigl(
\sum_{n=0}^{\infty}T_{\beta}^{n}(1) \, z^n
\Bigr)}
\end{equation}
where $N$, 
which depends upon $\beta$, 
is the minimal positive integer 
such that $T_{\beta}^{N}(1) = 0$;
in the case where
$T_{\beta}^{j}(1) \neq 0$ for all $j \geq 1$,
 ``$z^N$" has to be replaced by ``$0$".
Up to the sign,
the expansion of the power series of the denominator
\eqref{zetafunctionfraction00} is the Parry Upper function
$f_{\beta}(z)$ at $\beta$. It satisfies
\begin{equation}
\label{parryupperdynamicalzeta_ii}
(i)\quad
f_{\beta}(z) = - \frac{1 - z^N}{\zeta_{\beta}(z)}
\qquad
\mbox{in the first case},
\end{equation}
\begin{equation}
\label{parryupperdynamicalzeta_i}
(ii)\qquad f_{\beta}(z) = - \frac{1}{\zeta_{\beta}(z)} 
\qquad \mbox{in the second case},
\end{equation}
and, denoting by $t_1 , t_2 , \ldots \in 
\mathcal{A}_{\beta}$ the coefficients in
\begin{equation}
\label{fbetazetabeta}
-1 +t_1 z + t_2 z^2 + t_3 z^3 + \ldots
= f_{\beta}(z) 
= 
-(1 - \beta z)
\Bigl(
\sum_{n=0}^{\infty}T_{\beta}^{n}(1) \, z^n
\Bigr),
\end{equation}
$f_{\beta}(z)$ 
is such that $0 . t_1 t_2 t_3 \ldots$
is the R\'enyi
$\beta$-expansion of unity $d_{\beta}(1)$.
The Parry Upper function
$f_{\beta}(z)$ 
has no zero in
$\{z \in \cb :|z| \leq 1/\beta\}$ except $z=1/\beta$
which is a simple zero.
\end{theorem}
}

\textcolor{black}{
The total ordering $<$ on $(1, +\infty)$ is
uniquely in correspondence with the lexicographical
ordering $<_{lex}$ on R\'enyi expansions of 1
by the following Proposition, 
which is Lemma 3 in Parry
\cite{parry}.
}
\textcolor{black}{
\begin{proposition}
\label{variationbasebeta}
Let $\alpha > 1$ and $\beta > 1$. 
If the R\'enyi $\alpha$-expansion of 1 is
$$d_{\alpha}(1) = 0. t'_1 t'_2 t'_3\ldots, 
\qquad ~i.e.
\quad
1 ~=~ \frac{t'_1}{\alpha} + \frac{t'_2}{\alpha^2} + \frac{t'_3}{\alpha^3} + \ldots$$
and the R\'enyi $\beta$-expansion of 1 is
$$d_{\beta}(1) = 0. t_1 t_2 t_3\ldots, 
\qquad ~i.e. 
\quad
1 ~=~ \frac{t_1}{\beta} + \frac{t_2}{\beta^2} + \frac{t_3}{\beta^3} + \ldots,$$
then $\alpha < \beta$ if and only if $(t'_1, t'_2, t'_3, \ldots) 
<_{lex} (t_1, t_2, t_3, \ldots)$. 
\end{proposition}
}

\textcolor{black}{
\begin{theorem}
\label{zeronzeron}
Let $n \geq 2$. A real number 
$\beta \in ( 1, \frac{1+\sqrt{5}}{2}\, ]$ 
belongs to   
$[\theta_{n+1}^{-1} , \theta_{n}^{-1})$ if and only if the 
R\'enyi $\beta$-expansion of unity 
$d_{\beta}(1)$
is of the form
\begin{equation}
\label{dbeta1nnn}
d_{\beta}(1) = 0.1 0^{n-1} 1 0^{n_1} 1 0^{n_2} 1 0^{n_3} \ldots,
\end{equation}
with $n_k \geq n-1$ for all $k \geq 1$.  
\end{theorem}
}

\textcolor{black}{
\begin{proof} 
Since $d_{\theta_{n+1}^{-1}}(1) = 0.1 0^{n-1} 1$ and
$d_{\theta_{n}^{-1}}(1) = 0.1 0^{n-2} 1$, 
Proposition \ref{variationbasebeta} implies that
the condition is sufficient. It is also necessary:
$d_{\beta}(1)$ begins as $0.1 0^{n-1} 1$
for all $\beta$ such that
$\theta_{n+1}^{-1} \leq \beta < \theta_{n}^{-1}$.
For such $\beta$s 
we write $d_{\beta}(1) = 0.1 0^{n-1} 1 u$~  
with digits in the alphabet
$\mathcal{A}_{\beta}
=\{0, 1\}$ common to all $\beta$s, that is
$$u= 1^{h_0} 0^{n_1} 1^{h_1} 0^{n_2} 1^{h_2} \ldots$$
and $h_0, n_1, h_1, n_2, h_2, \ldots$ integers $\geq 0$.
The Conditions of Parry (\cite{lothaire} Chap. 7)
applied to the sequence
$(1, 0^{n-1}, 1^{1+h_0}, 0^{n_1}, 1^{h_1}, 0^{n_2}, 1^{h_3}, \ldots)$,
which characterizes uniquely the base of numeration $\beta$, 
readily implies
$h_0 = 0$ and 
$~h_k = 1$ and 
$n_k \geq n-1$ for all $k \geq 1$.
\end{proof}
The polynomials of the class $\mathcal{B}$ are all
the polynomial sections of the power series
$f_{\beta}(z)$ for $\beta$ in the interval 
$(1, (1+\sqrt{5})/2)$. Indeed, from
\eqref{dbeta1nnn} 
and \eqref{fbetazetabeta}, the power series
in \eqref{fbetazetabeta} takes the form:
$$-1 + x + x^n + x^{m_1}+x^{m_2}+\ldots+x^{m_s}+\ldots$$
with the distanciation conditions:
$m_1 - n \geq n-1$, 
$m_{q+1}-m_q \geq n-1$ for $1 \leq q$.
}

\section*{Acknowledgements}

We would like to thank Florent Jouve for a very helpful 
discussion and the referees for their useful comments.

\end{document}